\documentclass[preprint,12pt]{elsarticle1}

\usepackage{subfig}
\usepackage{graphicx}
\usepackage{caption,color}   
\usepackage{amssymb}
\usepackage{amsthm}
\usepackage{amsmath}
\usepackage{hyperref}
\usepackage{epic}
\usepackage{setspace}
\usepackage{float}
\usepackage{multirow}
\newtheorem{thm}{Theorem}[section]

\newtheorem{lem}[thm]{Lemma}
\newtheorem{pro}[thm]{Proposition}

\numberwithin{equation}{section}
\journal{}
\begin{document}
\begin{spacing}{1.15}
\begin{frontmatter}
\title{\textbf{\\ On the  second-largest modulus among the eigenvalues of a power hypergraph}}
\author[label1]{Changjiang Bu}\ead{buchangjiang@hrbeu.edu.cn}
\author[label2]{Lixiang Chen}\ead{clxmath@163.com}
\author[label2]{Yongtang Shi}\ead{shi@nankai.edu.cn}

\address[label1]{College of Mathematical Sciences, Harbin Engineering University, Harbin, PR China}
\address[label2]{Center for Combinatorics, LPMC, Nankai University, Tianjin, PR China}


\begin{abstract}
It is well known that the algebraic multiplicity of an eigenvalue of a graph (or real symmetric matrix) is equal to the dimension of its corresponding linear eigen-subspace, also known as the geometric multiplicity.
However, for hypergraphs, the relationship between these two multiplicities remains an open problem.
For a graph $G=(V,E)$ and $k \geq 3$, the $k$-power hypergraph $G^{(k)}$ is a $k$-uniform hypergraph obtained by adding $k-2$ new vertices to each edge of $G$, who always has non-real eigenvalues.
In this paper, we determine the second-largest modulus $\Lambda$ among the eigenvalues of $G^{(k)}$, which is indeed an eigenvalue of $G^{(k)}$.
The projective eigenvariety $\mathbb{V}_{\Lambda}$ associated with  $\Lambda$ is the set of the eigenvectors of $G^{(k)}$ corresponding to $\Lambda$ considered in the complex projective space.
We show that the dimension of $\mathbb{V}_{\Lambda}$ is zero, i.e, there are finitely many eigenvectors corresponding to $\Lambda$ up to a scalar.
We give both the algebraic multiplicity of $\Lambda$ and the total multiplicity of the eigenvector in $\mathbb{V}_{\Lambda}$ in terms of the number of the weakest edges of $G$. 
Our result show that these two multiplicities are equal.
\end{abstract}
%
\begin{keyword} eigenvalue, eigenvector, multiplicity, signed graph, power hypergraph \\
\emph{AMS classification(2020):} 05C50, 05C65.
\end{keyword}
\end{frontmatter}

\section{Introduction}
Let $\rho(G)$ and $\Lambda(G)$ denote the largest and second-largest moduli among  eigenvalues of a graph $G$, both of which have been extensively studied in various fields.
From the Perron-Frobenius Theorem, the largest modulus $\rho(G)$, also known as the spectral radius, is an eigenvalue of $G$ with algebraic multiplicity 1 if $G$ is connected.
The spectral radius $\rho(G)$ is a key quantity determining various different dynamical processes on $G$ \cite{Mieghem2009}.
Restrepo et al. introduced and investigated the \emph{dynamical importance} of an edge $e$, defined as the decrease of the spectral radius caused by its removal, i.e., $\rho(G) - \rho(G-e)$, and applied their results to real networks \cite{PhysRevLett.97.094102}.
The edges with the maximum and minimum dynamical importance are referred to  as the \emph{strongest edges} and \emph{weakest edges} of $G$, respectively.
The strongest edges of a graph have been widely studied \cite{2015147,PhysRevE.84.016101}, while this paper provides results concerning the weakest edges of a graph.

Evaluating the value of $\Lambda(G)$ and the algebraic multiplicity of the second-largest eigenvalues of $G$ are fairly challenging problems and have some applications in other mathematical fields.
For an $r$-regular connected graph $G$, if $\Lambda(G)$ is small compared with $r$ then $G$ is a good expander \cite[Page 68]{cvetkovic1980spectra}.
Specifically,  $G$ is called a Ramanujan graph if $\Lambda(G) \leq 2\sqrt{r-1}$ (see \cite{Murty2003}).
Jiang et al. evaluated  the algebraic multiplicity of the second-largest eigenvalues of graph with bounded maximum degree, and applied this result to solve a longstanding problem on the maximum number of equiangular lines in high dimensional euclidean space \cite{zi2021}.

Many authors have attempted to develop a theory on the spectral radius and the second-largest eigenvalue of hypergraphs, as it has applications in some fields \cite{chung1993laplacian,feng1996spectra,Friedman1995second,Li2019first}.
A \emph{$k$-uniform hypergraph} is a generalization of a graph in which the edges are $k$-element subsets of the set of vertices.
\emph{Tensors} (also called hypermatrices) are natural generalizations of matrices.
A matrix can be represented as an array indexed by two subscripts,  whereas a  $k$-order tensor is indexed by $k$ subscripts.
The spectrum of \emph{the adjacency tensor} of a hypergraph  is called \emph{the spectrum of the hypergraph} \cite{cooper2012spectra}.
The \emph{algebraic multiplicity} of an eigenvalue $\lambda$ refers to the number of times $\lambda$ appears as a root of the characteristic polynomial.
The \emph{eigenvariety} $\mathcal{V}_{\lambda}$ associated with $\lambda$ is the set of eigenvectors corresponding to $\lambda$ together with zero, and the \emph{projective eigenvariety} $\mathbb{V}_{\lambda}$ is the set of eigenvectors corresponding to $\lambda$ consider in the complex projective space.
Due to the nonlinearity of the eigenvalue equations of tensors,  the variety $\mathcal{V}_{\lambda}$ is generally not a linear subspace.
The relationship between the algebraic multiplicity of eigenvalues and the dimension of the eigenvariety is well understood for graphs (or matrices), as it follows directly from classical linear algebra.
However, for hypergraphs (or tensors), this relationship remains an open problem,
as discussed in the Hu-Ye's conjecture \cite{Hu2016}, the Fan's conjecture \cite{fan2024}, and the verifications of these conjectures  in some cases \cite{Cooper2022,Zheng2024}.

For a $k$-uniform hypergraph $H$, the spectral radius $\rho(H)$ is an eigenvalue of $H$ as shown by the Perron-Frobenius Theorem for tensors \cite{Chang2008}.
Fan et al. proved that the dimension of the projective eigenvariety $\mathbb{V}_{\rho}(H)$ is zero if $H$ is connected \cite{fan2019dimension}, and determined the size $|\mathbb{V}_{\rho}(H)|$ of $\mathbb{V}_{\rho}(H)$ \cite{Fan2019}.
The $k$-power hypergraph $G^{(k)}$ is the $k$-uniform hypergraph that is obtained by adding $k-2$ new vertices to each edges of a graph $G$ for $k \geq 3$, which is also known as the expansion of  $G$ in extremal hypergraph theory.
Chen et al. characterized all of the eigenvalues of $G^{(k)}$ in terms of the eigenvalues of signed subgraphs of $G$ \cite{CHEN2023205}. They also determined the algebraic multiplicity $\mathrm{am}_{\rho}(G^{(k)})$ of $\rho(G^{(k)})$ \cite{CHEN2024105909}.
Notably, Chen et al.'s results, together with those of Fan et al., led to the conclusion that
$\mathrm{am}_{\rho}(G^{(k)}) = |\mathbb{V}_{\rho}(G^{(k)})|$ as shown in \cite{fan2024}.

In this paper, we determined the second-largest modulus $\Lambda$ among eigenvalues of $G^{(k)}$.
More precisely, for $k \geq 4$, we show that $\Lambda=\sqrt[k]{\rho(G-e)^2}$, where $e$ is the weakest edge of $G$ (see Theorem \ref{eigenvalue}). For $k=3$, we determine $\Lambda$ in all distinct cases (see Theorem \ref{thm4}).
For a connected graph $G\neq K_2$ and $k \geq 4$,
we show that  the dimension of the projective eigenvariety $\mathbb{V}_{\Lambda}(G^{(k)})$ is zero (see Theorem \ref{finite}), i.e, there are finitely many eigenvectors of $G^{(k)}$ corresponding to $\Lambda$ up to a scalar.
For the zero-dimensional projective eigenvariety $\mathbb{V}_{\Lambda}(G^{(k)})$, we use $m(\mathbf{p})$ to denote the multiplicity of point $\mathbf{p}$ in $\mathbb{V}_{\Lambda}(G^{(k)})$.
Let $\#\mathbb{V}_{\Lambda}(G^{(k)})$ be the total multiplicity of its eigenvector in the $\mathbb{V}_{\Lambda}(G^{(k)})$, i.e., $\#\mathbb{V}_{\Lambda}(G^{(k)})=\sum_{\mathbf{p} \in \mathbb{V}_{\Lambda}(G^{(k)}) } m(\mathbf{p})$.
In terms of the number of weakest edges of $G$, we give both the algebraic multiplicity $\mathrm{am}_{\Lambda}(G^{(k)})$ and $\#\mathbb{V}_{\Lambda}(G^{(k)})$, thereby showing that $\mathrm{am}_{\Lambda}(G^{(k)})=\#\mathbb{V}_{\Lambda}(G^{(k)})$ (see Theorem \ref{multiplicity}).

\section{Preliminaries}

In this section, we introduce some basic notation and give auxiliary lemmas on the spectra of signed graphs and hypergraphs.

\subsection{The spectra of signed graphs }

A \emph{signed graph} $G_\pi$ is a pair $(G, \pi)$, where $G=(V,E)$ is a graph and $\pi:E \rightarrow \{+1,-1\}$ is the edge sign function.
We use $i \sim j$ to denote that the vertices $i$ and $j$ are adjacent in the graph $G$.
The adjacency matrix $A(G_\pi)=(A_{ij})$ of the signed graph $G_\pi$ is the symmetric $\{0,+1,-1\}$-matrix, where
\begin{align*}
A_{ij}=\left\{ \begin{array}{l}
 \pi(i,j),{\kern 35pt}  i \sim j, \\
 0, {\kern 57pt}\mathrm{ otherwise}. \\
\end{array} \right.
\end{align*}
The eigenvalues of $A(G_{\pi})$ are called the eigenvalues of $G_{\pi}$.


For a signed graph $G_{\pi}$ with $n$ vertices, let $\lambda_1(G_{\pi}) \geq \lambda_2(G_{\pi})\geq \cdots \geq \lambda_n(G_{\pi})$ be the eigenvalues of $G_{\pi}$.
The Cauchy Interlacing Theorem holds for principal sub-matrices of real symmetric matrices \cite[Theorem 1.3.11]{cve2010spectra}, so it is valid in signed graphs.
For any vertex $v$ of $G_{\pi}$, we have that
$\lambda_{1}(G_{\pi}) \geq  \lambda_{1}(G_{\pi}-v) \geq \lambda_{2}(G_{\pi}) \geq \lambda_{2}(G_{\pi}-v)\geq \cdots \geq \lambda_{n-1}(G_{\pi}-v) \geq \lambda_{n}({G_{\pi}})$.
Note that  $\lambda_{1}(G_{\pi})$ is not necessarily equal to the spectral radius $\rho(G_{\pi})$. In fact, $\rho(G_{\pi})=\max\{|\lambda_i(G_{\pi})|: i \in [n]\}=\max\{\lambda_1(G_{\pi}),-\lambda_{n}({G_{\pi}})\}$.

We call signed graphs $G_{\pi}$ and $G_{\pi'}$ \emph{switching equivalent} if there is a diagonal matrix $D$ with diagonal entries $\pm 1$ such that $A(G_{\pi'})=D^{-1}A(G_{\pi})D$. Clearly, switching equivalent signed graphs have the same spectrum. The signed graphs $G_{+}$ and $G_{-}$ are the ones with all signs $+1$ and all signs $-1$, respectively.
The spectral radius of a signed graphs $G_{\pi}$  does not exceed the spectral radius of its underlying graph $G$ \cite{CHEN2024105909}.

\begin{lem}\cite[Lemma 2.6]{CHEN2024105909}\label{lemma1}
Let $G$ be a connected graph. Then $\rho(G_{\pi}) \leq \rho(G)$, with equality if and only if $G_{\pi}$ is switching equivalent to $G_{+}$ or $G_{-}$.
\end{lem}

\begin{lem}\cite[Fact 2]{STANIC201980}\label{lem2.2}
For a real vector $\mathbf{x}=(x_i)$, let the vector $\mathbf{x}^*=(|x_i|)$.
For any eigenpair $(\lambda,\mathbf{x})$ of a signed graph $G_{\pi}$, there exists a switching equivalent signed graph of $G_{\pi}$ with eigenpair $(\lambda, \mathbf{x}^*)$.
\end{lem}

If $G_{\pi}$ is switching equivalent to $G_{+}$, then $G_{\pi}$ is called a \emph{balanced} signed graph.
Stani\'{c} \cite{STANIC201980} gave the following upper bound for the largest eigenvalue of $G_{\pi}$.
\begin{lem}\cite[Theorem 3.1]{STANIC201980}\label{lemma2}
For any signed graph $G_{\pi}$, there exists a balanced spanning subgraph $H_{\tilde{\pi}}$ of $G_{\pi}$ such that  $\lambda_1(G_{\pi}) \leq \lambda_1(H_{\tilde{\pi}}) = \rho(H)$.
\end{lem}


Let $G$ be a connected graph. If $G_{\pi}$ is balanced, it is clear that the spanning subgraph $H$ is trivial in Lemma \ref{lemma1}, i.e., $H=G$.
If $G_{\pi}$ is not balanced, we can get the following observation.

\begin{pro}\label{pro7}
Let $G$ be a connected graph.
If $G_{\pi}$ is not balanced, there exists a spanning proper subgraph $H$ of $G$ such that $\lambda_1(G_{\pi}) < \rho(H)$.
\end{pro}

\begin{proof}
Let $\mathbf{x}=(x_v)$ be a unit eigenvector corresponding to $\lambda_1(G_{\pi})$.
Without loss of generality, we  assume that the entries  of $\mathbf{x}$ are non-negative; otherwise, we consider the switching equivalent signed graph of $G_{\pi}$ mentioned in Lemma \ref{lem2.2}.
Since $G_{\pi}$ is not balanced, then $G_{\pi}$ contains at least one negative edge.
Let $H$ be the spanning proper subgraph obtained from $G$ by deleting all negative edges of $G_{\pi}$. 
We use ${i \overset{+}{\sim} j}$ (resp. ${i \overset{-}{\sim} j}$) to denote the vertices $i$ and $j$ are adjacent with positive (resp. negative) edges in $G_{\pi}$.
Then we have that
\begin{align}\label{eq1}
\lambda_1(G_{\pi}) =2 \sum_{i \overset{+}{\sim} j}x_ix_j - 2\sum_{i \overset{-}{\sim} j}x_ix_j  \leq 2 \sum_{i \overset{+}{\sim} j}x_ix_j \leq \rho(H).
\end{align}
Assume that all  equalities in \eqref{eq1} hold, then $x_ix_j>0$ when ${i \overset{+}{\sim} j}$ and $x_ix_j=0$ when ${i \overset{-}{\sim} j}$.
Let $V_0=\{v \in V(G) :x_v=0\}$.
Thus, an edge in $G_{\pi}$ is negative if and only if it is incident to a vertex in $V_0$.
Let $D_{V_0}=(d_{ii})$ be a diagonal matrix, where $d_{ii}=-1$ if $i \in V_0$ and $d_{ii}=1$ if $i \notin V_0$.
It is following that $A(G_+)=D_{V_0}^{-1}A(G_{\pi})D_{V_0}$, which contradicts that $G_{\pi}$ is not balanced.
\end{proof}

Let $\Pi$ denote the set of all sign functions on $E$. Let $\Gamma(G)$ denote the set of signed graphs on $G$ with the spectral radii  less than $\rho(G)$, i.e., let
\begin{align*}
\Gamma(G)=\{G_{\pi}: \mbox{$\pi  \in \Pi$ such that $\rho(G_{\pi})<\rho(G) $} \}.
\end{align*}

\begin{pro}\label{pro2}
Let $G=(V,E)$ be a connected graph. The set $\Gamma(G)$ is empty if and only if $G$ is a tree or an odd-unicyclic graph.
\end{pro}
\begin{proof}
Let $\mathcal{D}$ be the set of all $|V | \times |V |$ diagonal matrices with diagonal entries $\pm1$.
Let $\Pi_{+}$ and $\Pi_{-}$ denote the sets of signed function on $E$  such that $A(G_{\pi})=D^{-1}A(G_{+})D$ and $A(G_{\pi})=D^{-1}A(G_{-})D$ for some $D \in \mathcal{D}$.
Since $G$ is connected, $D^{-1}A(G)D=A(G)$ implies that $D={I}$ or $D={-I}$. It implies that there is a two-one correspondence between $\mathcal{D}$ and $\Pi_{+}$ and
similarly between $\mathcal{D}$ and $\Pi_{-}$. Thus
\begin{equation*}
  |\Pi_{+}|=|\Pi_{-}|=\frac{|\mathcal{D}|}{2}=2^{|V|-1}.
\end{equation*}

Since the set $\Gamma(G)$ is empty, we have $\rho(G_{\pi})=\rho(G)$ for all $\pi \in \Pi$.
From Lemma \ref{lemma1}, we see that $|\Pi|=|\Pi_{+}|=|\Pi_{-}|$ if $G$ is bipartite and $|\Pi|=|\Pi_{+}| + |\Pi_{-}|$ if $G$ is not bipartite. It implies that the number of signed functions is
\begin{align*}
 2^{|E|}=\begin{cases}
             2^{|V|-1}, & \mbox{if $G$ is bipartite, } \\
             2^{|V|}, & \mbox{otherwise}.
            \end{cases}
\end{align*}
Hence, $G$ is a tree or a unicyclic graph with an odd-cycle.
\end{proof}

Let $G(m)$ be the set of all connected un-labeled subgraphs of $G$ with at most $m$ edges.
For $\widehat{G} \in G(m)$, let $N_{G}(\widehat{G})$ denote the number of subgraphs of $G$ isomorphic to $\widehat{G}$.
A \emph{parity-closed walk} in $G$ is a closed walk that uses each edge an even number of times.
Let ${P_{d}(G)}$ be the number of parity closed walks of length $d$ in $G$.
A closed walk in $\widehat{G}$ is called \emph{covering} if it uses each edge at least once.
Let $p_{\ell}(\widehat{G})$ be the number of covering parity-closed of  length $\ell$ in $\widehat{G}$.
We observe that ${P_{d}(G)}$  has a natural decomposition as
\begin{align}\label{eq6}
 P_{d}(G)= \sum_{\widehat{G} \in G(\frac{d}{2})} p_{d}(\widehat{G})N_{G}(\widehat{G}).
\end{align}

The $d$-th order \emph{spectral moment} $\mathrm{S}_d(G)$ of a graph $G$ is the sum of $d$-th powers of all eigenvalues of $G$,
which also applies to signed graphs and hypergraphs.
It has been demonstrated that $P_d(G)$ is the arithmetic mean of the spectral moments of signed graphs with underlying graph $G$ \cite{CHEN2024105909}.

\begin{lem} \cite[Theorem 3.1]{CHEN2024105909}\label{lem5}
  Let $G=(V,E)$ be a connected graph. Then
 \begin{align*}
 P_{d}(G)= 2^{-|E|}\sum_{\pi \in \Pi} \mathrm{S}_{d}(G_{\pi}) .
\end{align*}
\end{lem}

Note that  $\lim_{\ell\rightarrow\infty}\frac{\sum_{\pi \in \Pi} \mathrm{S}_{2\ell}(G_{\pi})}{\rho(G)^{2\ell}}$ is the total multiplicity of eigenvalues of all signed graphs  on  $G$ whose modulus is equal to $\rho(G)$.
Based on  Lemma \ref{lem5}, the following result is derived in \cite{CHEN2024105909}.

\begin{lem}\cite[Lemma 7.1]{CHEN2024105909}\label{lem4}
  Let $G=(V,E)$ be a connected graph. Then $$\lim_{\ell\rightarrow\infty}\frac{P_{2\ell}(G)}{\rho(G)^{2\ell}}=\lim_{\ell\rightarrow\infty}\frac{p_{2\ell}(G)}{\rho(G)^{2\ell}}=2^{|V|-|E|}.$$
\end{lem}

\subsection{The spectra of hypergraphs}

For a positive integer $n$, let $\left[ n \right] = \left\{ {1, \ldots ,n} \right\}$.
A $k$-order $n$-dimensional complex tensor $T= \left( {{t_{{i_1} \cdots {i_k}}}} \right) $ is a multidimensional array with $n^k$ entries in complex number field $\mathbb{C}$, where ${i_j} \in \left[ n \right]$, $j = 1, \ldots ,k$.
For $\mathbf{x} = {\left( {{x_1}, \ldots ,{x_n}} \right)^{\top}} \in {\mathbb{C}^n}$,  let  ${\mathbf{x}^{\left[ {k - 1} \right]}} = {\left( {x_1^{k - 1}, \ldots ,x_n^{k - 1}} \right)^{\top}}$. The $i$-th component of the vector $T{\mathbf{x}^{k - 1}} \in \mathbb{C}^{n}$ is defined as
\[{\left( {T{\mathbf{x}^{k - 1}}} \right)_i} = \sum\limits_{{i_2}, \ldots ,{i_k}=1}^n {{a_{i{i_2} \cdots {i_k}}}{x_{{i_2}}} \cdots {x_{{i_k}}}} .\]
If there exists a nonzero vector $\mathbf{x}$ such that $T{\mathbf{x}^{k - 1}} = \lambda {\mathbf{x}^{\left[ {k - 1} \right]}}$, then $\lambda$ is called an \emph{eigenvalue} of $T$ and $\mathbf{x}$ is an \emph{eigenvector} of $T$ corresponding to $\lambda$ \cite{lim2005singular,qi2005eigenvalues}.
The \emph{characteristic polynomial} of $T$ is defined as the resultant of the polynomial system $(\lambda \mathbf{x}^{[k-1]}-T\mathbf{x}^{k-1})$ \cite{qi2005eigenvalues}.
The algebraic multiplicity of an eigenvalue is the multiplicity as a root of the characteristic polynomial.

A hypergraph $H=(V_H,E_H)$ is called \emph{$k$-uniform} if each edge of $H$ contains exactly $k$ vertices.
Similar to the relation between graphs and matrices, there is a natural correspondence between uniform hypergraphs and tensors.
For a $k$-uniform hypergraph $H$ with $n$ vertices, its \emph{adjacency tensor} ${A}_H=(a_{i_1i_2\ldots i_k})$ is a $k$-order $n$-dimensional tensor, where
\[{a_{{i_1}{i_2} \ldots {i_k}}} = \begin{cases}
                                    \frac{1}{{\left( {k - 1} \right)!}}, & \mbox{if $ \left\{ {{i_1},{i_2},\ldots ,{i_k}} \right\} \in {E_H}$}, \\
                                    0, & \mbox{otherwise}.
                                  \end{cases}\]
When $k=2$, ${A}_H$ is the usual adjacency matrix of the graph $H$.
The spectrum of the adjacency tensor $A_H$ is called the spectrum of hypergraph $H$.

Set $V_H=[n]$.
For $\mathbf{x}=(x_1,x_2, \ldots,x_n)^{\top}$, and let $x^{S}=\prod_{s\in S}{x_s}$ for $S \subseteq V_H$.
By $E_i(H)$, we denote the set of hyperedges containing $i$.
Let $\lambda$ be an eigenvalue of $H$, and let $\mathcal{V}_{\lambda}=\mathcal{V}_{\lambda}(H)$ be the set of all eigenvectors of $H$ corresponding to $\lambda$ together with zero, i.e.
\begin{align}\label{eqvariety}
 \mathcal{V}_{\lambda}=\{\mbox{$\mathbf{x} \in \mathbb{C}^n: \sum_{e\in E_i(H) } x^{e \setminus\{i\}}=\lambda x_i^{k-1}$ for all $i \in V_H$}\}.
\end{align}
Observe that the system of equations in \eqref{eqvariety}  is not linear yet for $k \geq 3$, and therefore $\mathcal{V}_{\lambda}$ is not a linear subspace of $\mathbb{C}^n$ in general. In fact, $\mathcal{V}_{\lambda}$ forms an affine variety in $\mathbb{C}^n$ \cite{usingalgebraic}.
The projective eigenvariety of $H$ associated with $\lambda$ is defined to be the projective variety (\cite{fan2019dimension})
\begin{align*}
 \mathbb{V}_{\lambda}=\{\mbox{$\mathbf{x} \in \mathbb{P}^{n-1}: \sum_{e\in E_i(H) } x^{e \setminus\{i\}}=\lambda x_i^{k-1}$ for all $i \in V_H$}\},
\end{align*}
where $\mathbb{P}^{n-1}$ is the complex projective spaces over $\mathbb{C}$ of dimension $n-1$.

Let $k \geq 3$.
Recall that the $k$-power hypergraph $G^{(k)}$ is the $k$-uniform hypergraph that is obtained by adding $k-2$ new vertices to each edges of a graph $G$.
All the distinct eigenvalues of  $G^{(k)}$ are derived by eigenvalues of signed subgraphs of $G$ as follows \cite{CHEN2023205}.
\begin{thm}\cite[Theorem 1.2]{CHEN2023205}\label{thm1}
The complex number $\lambda$ is an eigenvalue of $G^{(k)}$ if and only if
  \begin{enumerate}
\renewcommand{\labelenumi}{(\alph{enumi})}
  \item some signed induced subgraph of $G$ has an eigenvalue $\sigma$ such that $\sigma^2 = \lambda^k$, when $k = 3$;
  \item some signed subgraph of $G$ has an eigenvalue $\sigma$ such that $\sigma^2 = \lambda^k$, when $k \geq 4$;
\end{enumerate}
\end{thm}

Recall that the $d$-th order spectral moment $\mathrm{S}_d(H)$ of a hypergraph $H$ is the sum of $d$-th powers of all eigenvalues of $H$.
For a power hypergraph $G^{(k)}$, Chen et al. gave an expression for  $\mathrm{S}_d(G^{(k})$ that  involves the number of subgraphs and parity-closed walks in $G$ \cite{CHEN2024105909}.

\begin{thm}\cite[Proposition 4.8]{CHEN2024105909}\label{thm5}
Let $G=(V,E)$. For $k \geq 3$,  the $d$-th order spectral moments of the power hypergraphs $G^{(k)}$ is
  \begin{equation*}
    \mathrm{S}_d(G^{(k})=\begin{cases}
                           \sum\limits_{\widehat{G} \in G(\frac{d}{k})} g(\widehat{G},k)p_{\frac{2d}{k}}(\widehat{G})N_{G}(\widehat{G}), & \mbox{$k\mid d$, }  \\
                           0, & \mbox{$k \nmid d$,}
                         \end{cases}
  \end{equation*}
where  $g(\widehat{G},k)=2^{|E(\widehat{G})|-|V(\widehat{G})|}(k-1)^{|V|-|V(\widehat{G})|+(k-2)(|E|-|E(\widehat{G})|)}k^{|V(\widehat{G})|+|E(\widehat{G})|(k-3)}$.
\end{thm}

The spectrum of a hypergraph is  \emph{$k$-symmetric} if it is invariant under a rotation of an angle $2\pi/k$ in the complex plane.
Indeed, the spectrum of a $k$-power hypergraph is $k$-symmetric \cite{fan2019spectral}.
By the symmetry of the spectrum of a power hypergraph and Theorem \ref{thm1}, we get the following property of the algebraic  multiplicities of  eigenvalues of the power hypergraphs.
\begin{pro}\label{lem6}
Let $k \geq 3$. Let $\lambda$ be an eigenvalue of $G^{(k)}$. Then $|\lambda|$ is an eigenvalue of $G^{(k)}$, and
the total algebraic  multiplicity of eigenvalues with the modulus $|\lambda|$ is $k$ times the algebraic  multiplicity of $\lambda$.
\end{pro}

From the Perron-Frobenius Theorem for tensors \cite{Chang2008}, it is known that the spectral radius $\rho(H)$ is an eigenvalue of $H$. 
Chen et al. used the spectral moment of $G^{(k)}$ to give the algebraic multiplicity of the spectral radius  \cite{CHEN2024105909}. 

\begin{thm}\cite[Theorem 7.2]{CHEN2024105909}\label{thm2}
Let $k\geq3$. For a connected graph $G=(V,E)$, the algebraic multiplicity of the spectral radius of $G^{(k)}$ is $k^{|E|(k-3)+|V|-1}$.
\end{thm}

For a connected uniform hypergraph $H$, the projective variety $\mathbb{V}_{\rho}(H)$ is characterized by the Smith normal form of the incidence matrix of $H$ over $\mathbb{Z}_k$ \cite{fan2024,Fan2019,fan2019dimension}.
Here are some related lemmas used in this paper.
\begin{lem}\cite{fan2024,Fan2019,fan2019dimension}\label{xinlemma}
Let $H$ be a connected uniform hypergraph.
 \begin{enumerate}
\renewcommand{\labelenumi}{(\alph{enumi})}
  \item \cite[Theorem 3.8]{fan2019dimension} The dimension of the projective eigenvariety $\mathbb{V}_{\rho} (H)$ is zero, i.e., there are finitely many eigenvectors of $H$ corresponding to $\rho(H)$ up to a scalar.
  \item \cite[Corollary 3.4]{fan2024} The multiplicity of each point in $\mathbb{V}_{\rho} (H)$ is 1.  
  \item \cite[Corollary 4.2]{Fan2019} Let $G=(V,E)$ be a connected graph. Then $\mathrm{am}_{\rho}(G^{(k)})=|\mathbb{V}_{\rho} (G^{(k)})|=k^{|E|(k-3)+|V|-1}$ for $k \geq 3$. 
\end{enumerate}
\end{lem}

\section{The characterization of the eigenpair $(\Lambda, \mathbf{x})$}

Let $G=(V,E)$. Our first goal is to characterize the second-largest modulus $\Lambda$ among the eigenvalues of the power hypergraph $G^{(k)}$.
We write $\rho_{V}(G)=\max_{v \in V} \rho(G-v)$ and $\rho_{E}(G)=\max_{e \in E} \rho(G-e)$.
It is clear that the weakest edge $e_0$ is the one such that $\rho(G-e_0)=\rho_{E}(G)$.
If $G$ is neither a tree nor an odd-unicyclic graph, from Proposition \ref{pro2}, $\Gamma(G)$ is not empty.
Thus, we  define
$$\rho_{\Gamma}(G)=\max_{G_{\pi} \in \Gamma(G) }\rho(G_{\pi} ).$$
By Lemmas  \ref{lemma1} and  \ref{lemma2}, we give the following result to compare $\rho_{V}(G)$, $\rho_{E}(G)$ and $\rho_{\Gamma}(G)$,
which plays a key role in determining $\Lambda$.

\begin{pro}\label{pro3}
 Let $G$ be a connected graph with $n$ vertices.
  \begin{enumerate}
\renewcommand{\labelenumi}{(\alph{enumi})}
  \item Then $\lambda_2(G)<\rho_{V}(G)\leq \rho_{E}(G)$.
  \item If $G$ is non-bipartite,  then $-\lambda_n(G)<\rho_{E}(G) $.
  \item If $\Gamma(G)$ is not empty, then  $\rho_{\Gamma}(G) <\rho_{E}(G)  $.
\end{enumerate}
\end{pro}

\begin{proof}
It is clear that (a) follows from the Interlacing Theorem for graphs.
Note that $\lambda_1(G_{-})=-\lambda_n(G)$ and $\rho(G_{\pi})=\max\{\lambda_1(G_{\pi}),\lambda_1(G_{-\pi})\}$.
If $G$ is non-bipartite, then $G_{-}$ is not balanced.
From  Proposition \ref{pro7}, there exists a spanning  proper subgraphs $H_1$ such that
$$ -\lambda_n(G)=\lambda_1(G_{-})< \rho(H_1) \leq \rho_{E}(G).$$
If $\Gamma(G)$ is not empty, for any $G_{\pi} \in \Gamma(G)$, we see that $G_{\pi}$ and $G_{-\pi}$ are not balanced by Lemma \ref{lemma1}.
From  Proposition \ref{pro7}, there exists a spanning  proper subgraphs $H_2$ such that
$$\rho(G_{\pi})=\max\{\lambda_1(G_{\pi}),\lambda_1(G_{-\pi})\} < \rho(H_2) \leq \rho_{E}(G) .$$
Then, we get (b) and (c).
\end{proof}

For the power hypergraph $G^{(k)}$ with $k \geq 4$,  we will show that second-largest modulus $\Lambda(G^{(k)})=\sqrt[k]{\rho_{E}(G)^2}$, which in turn implies that $\rho(G)-\sqrt[2]{\Lambda(G^{(k)})^k}$ serves as a tight lower bound for the dynamical importance of an edge in $G$.

\begin{thm}\label{eigenvalue}
Let $G$ be a connected graph with $n$ vertices.
Let $\Lambda=\Lambda(G^{(k)})$ be the second-largest modulus among eigenvalues of $G^{(k)}$.
Then  $\Lambda=\sqrt[k]{\max_{e \in E} \rho(G-e)^2}$ for $k \geq 4$, and $\Lambda$ is an eigenvalue of $G^{(k)}$.
\end{thm}

\begin{proof}
Let  $F$ be a proper subgraph of $G$. We have that  $\rho(F_{\pi}) \leq \rho(F)\leq \rho_{E}(G)$ for any $\pi$.
Thus, we observe from Theorem \ref{thm1} that $ \Lambda $ is determined by exactly four potential quantities: $ \rho_{E}(G) $, $ \lambda_2(G) $, $ -\lambda_n(G) $, and $ \rho_{\Gamma}(G) $.
Then we have that
\begin{align}\label{eq4}
 \Lambda^k =\begin{cases}
                                 \max\{\rho_{E}(G)^2, \lambda_2(G)^2\}, & \mbox{if $G$ is a tree,} \\
                                 \max\{\rho_{E}(G)^2, \lambda_2(G)^2,\lambda_n(G)^2\}, & \mbox{if $G$ is an odd-unicyclic graph,}  \\
                                 \max\{\rho_{E}(G)^2, \rho_{\Gamma}(G)^2, \lambda_2(G)^2\}, & \mbox{if $G$ is bipartite but not a tree,}  \\
                                 \max\{\rho_{E}(G)^2, \rho_{\Gamma}(G)^2, \lambda_2(G)^2,\lambda_n(G)^2\}, & \mbox{otherwise}.
                               \end{cases}
\end{align}
By  Proposition \ref{pro3}, we get $\Lambda^k=\rho_{E}(G)^2=\max_{e \in E} \rho(G-e)^2$. Proposition \ref{lem6} states that $\Lambda$ is an eigenvalue of $G^{(k)}$.
\end{proof}


For the case of $ k = 3$, Theorem \ref{thm1} shows that the eigenvalues of $ G^{(3)}$ are generated by the signed induced subgraphs (vertex-deleted subgraphs) of $ G $. It implies that the eigenvalue $\Lambda(G^{(3)})$ can be directly obtained by replacing $\rho_{V}(G) $ with $ \rho_{E}(G) $ in \eqref{eq4}.

\begin{thm}\label{thm4}
Let $G$ be a connected graph with $n$ vertices. Then
\begin{equation*}
  \Lambda(G^{(3)})=\begin{cases}
               \sqrt[3]{\rho_{V}(G)^2}, & \mbox{if $G$ is a tree,} \\
                                 \max\{\sqrt[3]{\rho_{V}(G)^2}, \sqrt[3]{\lambda_n(G)^2}\}, & \mbox{if $G$ is an odd-unicyclic graph,}  \\
                                 \max\{\sqrt[3]{\rho_{V}(G)^2},\sqrt[3]{\rho_{\Gamma}(G)^2}\}, & \mbox{if $G$ is bipartite but not a tree,}  \\
                                 \max\{ \sqrt[3]{\rho_{V}(G)^2},\sqrt[3]{\rho_{\Gamma}(G)^2},\sqrt[3]{\lambda_n(G)^2}\}, & \mbox{otherwise}.
             \end{cases}
\end{equation*}
\end{thm}
\noindent \textbf{Remark:} In Appendix A, we provide some examples to demonstrate that all candidates in Theorem \ref{thm4} are potential.

For a graph $G=(V,E)$ and $e \in E$, we use $N_e$ to denote the set of added vertices of $G^{(k)}$ on the edge $e$. Thus, the set $e \cup N_e$ is a hyperedge of $G^{(k)}$.
Recall that $x^{S}=\prod_{s\in S}{x_s}$ for $S \subseteq V(G^{(k)})$.
By $E_i(G^{(k)})$, we denote the set of hyperedges containing $i$.
Then it follows that $(\lambda, \mathbf{x})$ is an eigenpair of $G^{(k)}$ if and only if
\begin{align}\label{noncore}
  \lambda x_i^{k-1}&= \sum_{h \in E_i(G^{(k)})}x^{h\setminus \{i\}} \notag \\
   & =\sum_{j:\{i,j\}\in E}x_jx^{N_{\{i,j\}}}
\end{align}
for every $i \in V$ and
\begin{align}\label{core}
  \lambda x_{v}^{k-1} & =x_ix_jx^{N_{\{i,j\}}\setminus \{v\}}
\end{align}
for every $v \in N_{\{i,j\}}$ and $\{i,j\} \in E$.

The vertices with degree one in a power hypergraph are called core vertices, which include both the original pendant vertices and the newly added vertices.
Next, we will show that the zero entries of the eigenvectors corresponding to $\Lambda$ indicate all the core vertices that lie on some weakest edge, which gives a way to identify the weakest edges of the graph using the eigenvectors of the hypergraph.

\begin{thm}\label{eigenvector}
Let $G \neq K_2$ be a connected graph.
Let $\Lambda$ be the second-largest modulus among eigenvalues of $G^{(k)}$ for $k \geq 4$.
For any eigenvector $\mathbf{x}=(x_v)$ corresponding to $\Lambda$,  there exists a weakest edge $e$ of $G$ such that the set $\{v \in V(G^{(k)}):x_v=0\}$
consists precisely of all the  core vertices that lie on  $e$.

\end{thm}

\begin{proof}

Note that $G$ has more than one edge, which implies that $\Lambda > 0$.
Let $\beta$ be such that $\beta^2=\Lambda^k$, and let $y_i$ be such that $y_i^2=x_i^k$ for $i \in V(G)$.
Consider the induced subgraph $\widehat{G}$ on the vertices $i\in V(G)$ with $x_i \neq 0$. Note that $x_i=0$ for all $i \in V(G)$ is impossible, because of \eqref{core} and $\Lambda \neq 0$.

Using \eqref{core}, we have that
\begin{align*}
  \Lambda^{k-2}(x^{N_{\{i,j\}}})^{k-1} & =\prod_{v \in N_{\{i,j\}}} \Lambda x_v^{k-1} \\
   & = \prod_{v \in N_{\{i,j\}}}x_ix_jx^{N_{\{i,j\}} \setminus \{v\}} \\
   & = (x_ix_j)^{k-2}(x^{N_{\{i,j\}}})^{k-3},
\end{align*}
that is, $(x^{N_{\{i,j\}}})^{k-3}(\Lambda^{k-2}(x^{N_{\{i,j\}}})^{2}-(x_ix_j)^{k-2})=0$, and hence
\begin{align*}
  \left(\beta x_ix_jx^{N_{\{i,j\}}}\right)^{k-3}\left((\beta x_ix_jx^{N_{\{i,j\}}})^2-(\Lambda y_iy_j)^2\right)=0.
\end{align*}
Therefore, we have that
\begin{align}\label{eq13}
\beta x_ix_jx^{N_{\{i,j\}}}=\mathrm{sgn}(i,j)(\Lambda y_iy_j),
\end{align}
where $\mathrm{sgn}(i,j) \in \{\pm 1,0\}$.
And, using \eqref{noncore}, we obtain that for every $i \in V(\widehat{G})$,
\begin{align*}
  \beta y_i& = \frac{\beta x_i}{\Lambda y_i}\Lambda x_i^{k-1}=\frac{\beta x_i}{\Lambda y_i}\sum_{j: \{i,j\} \in E(\widehat{G})}x_jx^{N_{\{i,j\}}}=\frac{1}{\Lambda y_i}\sum_{j: \{i,j\} \in E(\widehat{G})}\beta x_ix_jx^{N_{\{i,j\}}} ,
\end{align*}
that is,
\begin{align}\label{eq3.5}
\beta y_i =\sum_{j: \{i,j\} \in E(\widehat{G})}\mathrm{sgn}(i,j) y_j.
\end{align}
Thus, $(\beta, \mathbf{y})$ is an eigenpair of a spanning signed subgraph $S_{\pi}$ of $\widehat{G}$, where the sign function $\pi$ is defined by $\mathrm{sgn}(i,j)$.
By Theorem \ref{eigenvalue}, we have that $ \beta = \pm \sqrt[2]{\Lambda^k} = \pm \rho_E(G) $.
Without loss of generality, set $\beta = \rho_E(G) $.
Note that $\rho_{E}(G)=\max_{e \in E} \rho(G-e)$, so
there exists a weakest edge $e_0$ in $G$ such that $ S_{\pi} $ is switching equivalent to $ G - e_0 $.
Then, in \eqref{eq3.5}, $\mathrm{sgn}(i,j)=0$ if only if $\{i,j\}=e_0$.
By \eqref{eq13}, it follows that $ x^{N_{e}} = 0 $ if only if $e=e_0$.
Using \eqref{core}, $ x^{N_{e_0}} = 0 $ implies $ x_v = 0 $ for all $ v \in N_{e_0} $.

Note that the weakest edges $e_0$ is not a cut edge unless a pendant edge.
If $e_0$ is a pendant edge, let $u$ denote the pendant vertex on $e_0$. Using \eqref{noncore}, we get $x_u=0$. Thus, we get $ \{v \in V(G^{(k)}) : x_v = 0\} = N_{e_0} \bigcup \{u\} $.
If $e_0$ is non-pendant,  we get $ \widehat{G} = G $, which implies $ x_i \neq 0 $ for all $ i \in V(G) $.
Thus, we  have $ \{v \in V(G^{(k)}) : x_v = 0\} = N_{e_0} $.
We conclude that $\{v \in V(G^{(k)}):x_v=0\}$ consists precisely of all the  core vertices that lie on  $e_0$.
\end{proof}


Let $E_w$ be the set of all weakest edges of $G$.
For $e \in E_w$, define \begin{equation*}
  \delta=\delta(e)=\begin{cases}
           0, & \mbox{if  $e$ is a pendant edge,}  \\
           1, & \mbox{otherwise}.
         \end{cases}
\end{equation*}
We use $G_e$ to denote the subgraph of $G$ obtained by removing $e$ when $\delta =1$, and by removing $e$ along with the corresponding pendant vertex when $\delta =0$. Note that $G_e$ is connected.
Let $V(G^{(k)})=\{1,2,\ldots, n\}$. For $\mathbf{x}=(x_1,\ldots,x_n)^\top \in \mathbb{V}_{\Lambda}(G^{(k)})$, denote the support set of $\mathbf{x}$ by $supp(\mathbf{x})=\{i: x_i \neq 0\}$.
By Theorem \ref{eigenvector}, we see that  $supp(\mathbf{x})$ is the set of vertices of graph $G$ excluding the core vertices on the weakest edge $e$.
Thus, the subhypergraph of $G^{(k)}$ induced by $supp(\mathbf{x})$ is the power hypergraph $G_e^{(k)}$.
Let $\hat{\mathbf{x}}$ denote the subvector of $\mathbf{x}$ obtained truncating entry $x_i$ from $\mathbf{x}$ for $i \in supp(\mathbf{x})$.
Note that $\Lambda=\rho(G_e^{(k)})$ is the spectral radius of $G_e^{(k)}$ from Theorem \ref{eigenvalue},
hence $\hat{\mathbf{x}} \in \mathbb{V}_{\rho}(G_e^{(k)})$.
Lemma \ref{xinlemma}(a) tells us that the dimension of $\mathbb{V}_{\rho}(G_e^{(k)})$ is zero, thus we obtain the following theorem.

\begin{thm}\label{finite}
Let $G \neq K_2$ be a connected graph.
Let $\Lambda$ be the second-largest modulus among eigenvalues of $G^{(k)}$ for $k \geq 4$.
Then the dimension of $\mathbb{V}_{\Lambda}$ is zero, i.e., there are finitely many eigenvectors of $G^{(k)}$ corresponding to $\Lambda$ up to a scalar.
\end{thm}

\begin{proof}
Let $\mathbb{V}_{\rho}(G_e^{(k)})$ be the  projective eigenvarity associated with the spectral radius of $G_e^{(k)}$. Let $\mathbf{0}_{e}$ be the $(k-1-\delta)$-dimensional zero vector.
For any $e \in E_w$ and any $\mathbf{x}_e \in \mathbb{V}_{\rho}(G_e^{(k)})$, the vector $\mathbf{x}=\mathbf{x}_e\oplus \mathbf{0}_{e}$ is an eigenvector of $G^{(k)}$ corresponding to $\Lambda$, as one can easily check.
Then we find a map
\begin{align}\label{eq14}
  \phi: \bigcup_{e \in E_w}\mathbb{V}_{\rho}(G_e^{(k)}) \rightarrow \mathbb{V}_{\Lambda}, \mathbf{x}_{e} \mapsto \mathbf{x}.
\end{align}
For each $\mathbf{x} \in \mathbb{V}_{\Lambda}$,  recall that $\hat{\mathbf{x}}$ denote the subvector of $\mathbf{x}$ obtained truncating entry $x_i$ from $\mathbf{x}$ for $i \in supp(\mathbf{x})$.  There exists $e \in E_w$ such that $\hat{\mathbf{x}} \in \mathbb{V}_{\rho}(G_e^{(k)})$.
So, $\phi$ is a surjective map, indeed, $\phi$ is a bijective map. For each $e \in E_w$, $\mathbb{V}_{\rho}(G_e^{(k)})$ is finite set by Lemma \ref{xinlemma}(a), it is following that  $\mathbb{V}_{\Lambda}$ is finite.
\end{proof}



\section{The algebraic multiplicity of $\Lambda$ and the total multiplicity of its eigenvector}

Let $\mathrm{am}_{\Lambda}(G^{(k)})$ denote the algebraic multiplicity of the eigenvalue $\Lambda$ for $G^{(k)}$.
Theorem \ref{finite} shows that the dimension of $\mathbb{V}_{\Lambda}(G^{(k)})$ is zero.
For the zero-dimensional projective eigenvariety $\mathbb{V}_{\Lambda}(G^{(k)})$, we use $m_{\mathbb{V}_{\Lambda}}(\mathbf{p})$ to denote the multiplicity of point $\mathbf{p}$ in $\mathbb{V}_{\Lambda}(G^{(k)})$.
Let $\#\mathbb{V}_{\Lambda}(G^{(k)})$ denote the total multiplicity of eigenvectors of $G^{(k)}$ corresponding to $\Lambda$, i.e., $\#\mathbb{V}_{\Lambda}(G^{(k)})=\sum_{\mathbf{p} \in \mathbb{V}_{\Lambda}(G^{(k)}) } m_{\mathbb{V}_{\Lambda}}(\mathbf{p})$.
In this section, we express $\mathrm{am}_{\Lambda}(G^{(k)})$ and $\#\mathbb{V}_{\Lambda}(G^{(k)})$ in terms of the number of weakest edges of $G$, there by showing that $\mathrm{am}_{\Lambda}(G^{(k)})=\#\mathbb{V}_{\Lambda}(G^{(k)})$.
The main result of this section is shown as follows.

\begin{thm}\label{multiplicity}
For a connected graph $G=(V,E)$ with $|E|>1$, let $n_0$ and $n_1$ be the number of pendent and non-pendent weakest edges of $G$, respectively.
Let $\Lambda$ be the second-largest modulus among eigenvalues of $G^{(k)}$ for $k \geq 4$.
Denote the algebraic multiplicity of the eigenvalue $\Lambda$ by $\mathrm{am}_{\Lambda}(G^{(k)})$, and the total multiplicity of eigenvectors of $G^{(k)}$ in $\mathbb{V}_{\Lambda}(G^{(k)})$ by $\#\mathbb{V}_{\Lambda}(G^{(k)})$.
Then
\begin{equation*}
\mathrm{am}_{\Lambda}(G^{(k)})=\#\mathbb{V}_{\Lambda}(G^{(k)})=\sum_{\delta=0}^{1}f_{\delta}(k)n_{\delta},
\end{equation*}
where $f_{\delta}(k)=k^{|E|(k-3)+|V|-1}\left({(k-1)^{k-1-\delta}k^{2-k+\delta}-2^{\delta}}\right)$ for $\delta= 0, 1$.
\end{thm}

We will determine $\mathrm{am}_{\Lambda}=\mathrm{am}_{\Lambda}(G^{(k)})$ and $\#\mathbb{V}_{\Lambda}=\#\mathbb{V}_{\Lambda}(G^{(k)})$ separately.
First, we give the algebraic multiplicity $\mathrm{am}_{\Lambda}$ by the spectral moments of $G^{(k)}$.

\vspace{3mm}
\noindent{\textbf{Part 1 of Theorem \ref{multiplicity}}} $\mathrm{am}_{\Lambda}=\sum_{\delta=0}^{1}f_{\delta}(k)n_{\delta}$.
\begin{proof}
 Proposition \ref{lem6} and Theorem \ref{thm2} show  that there are $k^{|E|(k-3)+|V|}$ eigenvalues of $G^{(k)}$ whose modulus are equal to $\rho(G^{(k)})$, and there are $k\mathrm{am}_{\Lambda}$ eigenvalues of $G^{(k)}$ whose modulus are equal to $\Lambda(G^{(k)})$.
Recall that $\rho(G^{(k )})$ and $\Lambda(G^{(k)})$ are the largest and second-largest modulus of eigenvalues of $G^{(k)}$.
Then we have that
 \begin{align*}
     k\mathrm{am}_{\Lambda}& =\lim_{\ell\rightarrow\infty} \frac{\mathrm{S}_{k\ell}(G^{(k)})-k^{|E|(k-3)+|V|}\rho(G^{(k)})^{k\ell}}{\Lambda(G^{(k)})^{k\ell}}.
  \end{align*}
Note that $\rho(G^{(k)})^{k}=\rho(G)^{2}$ and $\Lambda(G^{(k)})^{k}={\rho_{E}(G)^{2}}$.
From Theorem \ref{thm5}, we have that
  \begin{align} \label{eq5}
     k\mathrm{am}_{\Lambda}=\lim_{\ell\rightarrow\infty}\frac{\sum_{\widehat{G} \in G (\ell)}g(\widehat{G},k)p_{2\ell}(\widehat{G})N_{G}(\widehat{G})-k^{|E|(k-3)+|V|}\rho(G)^{2\ell}}{\rho_{E}(G)^{2\ell}},
  \end{align}
where $g(\widehat{G},k)=2^{|E(\widehat{G})|-|V(\widehat{G})|}(k-1)^{|V|-|V(\widehat{G})|+(k-2)(|E|-|E(\widehat{G})|)}k^{|V(\widehat{G})|+|E(\widehat{G})|(k-3)}$.
For $\widehat{G} \in G(\ell)$ satisfying $\rho(\widehat{G})<\rho_{E}(G)$, from Lemma \ref{lem4}, we know that $p_{2\ell}(\widehat{G})$ has no contribution to \eqref{eq5}.
Let $\mathcal{G}_{\Lambda}=\{\widehat{G}\in G(\ell) :\mbox{$\rho(\widehat{G}) =\rho_E(G)$}\}$.
Then we have that
\begin{align}\label{eq7}
k\mathrm{am}_{\Lambda}& =\lim_{\ell\rightarrow\infty}\frac{\sum_{\widehat{G} \in \mathcal{G}_{\Lambda}}g(\widehat{G},k)p_{2\ell}(\widehat{G})N_{G}(\widehat{G})-k^{|E|(k-3)+|V|}\left(\rho(G)^{2\ell}-2^{|E|-|V|}p_{2\ell}(G)\right)}{\rho_{E}(G)^{2\ell}}.
\end{align}
By \eqref{eq6} and Lemma \ref{lem5}, we have that
\begin{align}\label{eq11}
&\lim_{\ell\rightarrow\infty}\frac{2^{|E|-|V|}p_{2\ell}(G)}{\rho_{E}(G)^{2\ell}} \notag \\
&=\lim_{\ell\rightarrow\infty}\frac{2^{|E|-|V|}P_{2\ell}(G)-\sum_{\widehat{G} \in \mathcal{G}_{\Lambda}}2^{|E|-|V|}p_{2\ell}(\widehat{G})N_{G}(\widehat{G})}{\rho_{E}(G)^{2\ell}}\notag\\
&=\lim_{\ell\rightarrow\infty}\frac{2^{-|V|}\sum_{\pi \in \Pi} \mathrm{S}_{2\ell}(G_{\pi}) -\sum_{\widehat{G} \in \mathcal{G}_{\Lambda}}2^{|E|-|V|}p_{2\ell}(\widehat{G})N_{G}(\widehat{G})}{\rho_{E}(G)^{2\ell}}.
\end{align}

Let $\Lambda_{\Pi}$ be the second-largest modulus among eigenvalues of $G_{\pi}$ for all $\pi \in \Pi$.
From Lemmas \ref{lem5} and \ref{lem4},  it is following that among all eigenvalues of $G_{\pi}$ for all $\pi \in \Pi$,
there are $2^{|V|}$ eigenvalues whose modulus are equal to $\rho(G)$.
It implies that the order of magnitude of $\rho(G)^{2\ell}-2^{-|V|}\sum_{\pi \in \Pi}\mathrm{S}_{2\ell}(G_{\pi})$ is determined by $\Lambda_{\Pi}^{2\ell}$ when $ \ell\rightarrow\infty$.
The possible values for $\Lambda_{\Pi}$ are $\lambda_{2}(G)$, $-\lambda_{|V|}(G)$ if $G$ is non-bipartite, and $\rho_{\Gamma}(G)$ if $\Gamma(G)$ is not empty.
From Proposition \ref{pro3}, we have $ {\rho_{E}(G)} > \Lambda_{\Pi} $ in any case.
Hence, we have
\begin{align}\label{eq12}
\lim_{\ell\rightarrow\infty}\frac{\rho(G)^{2\ell}-2^{-|V|}\sum_{\pi \in \Pi} \mathrm{S}_{2\ell}(G_{\pi}) }{\rho_{E}(G)^{2\ell}}=0.
\end{align}
By \eqref{eq11} and \eqref{eq12}, we have that
\begin{equation}\label{eq8}
\lim_{\ell\rightarrow\infty}\frac{\rho(G)^{2\ell}-2^{|E|-|V|}p_{2\ell}(G)}{\rho_{E}(G)^{2\ell}}=\lim_{\ell\rightarrow\infty}\frac{ \sum_{\widehat{G} \in \mathcal{G}_{\Lambda}}2^{|E|-|V|}p_{2\ell}(\widehat{G})N_{G}(\widehat{G})}{\rho_{E}(G)^{2\ell}}.
\end{equation}
Substituting \eqref{eq8} into \eqref{eq7}, we have that
\begin{equation}\label{eq9}
  k\mathrm{am}_{\Lambda} =\lim_{\ell\rightarrow\infty}\frac{\sum_{\widehat{G} \in \mathcal{G}_{\Lambda}}\left(g(\widehat{G},k)-2^{|E|-|V|}k^{|E|(k-3)+|V|}\right)p_{2\ell}(\widehat{G})N_{G}(\widehat{G})}{\rho_{E}(G)^{2\ell}}.
\end{equation}
For any ${\widehat{G} \in \mathcal{G}_{\Lambda}}$, Lemma \ref{lem4} tells us that $$\lim_{\ell\rightarrow\infty}\frac{p_{2\ell}(\widehat{G})}{\rho_{E}(G)^{2\ell}}=\lim_{\ell\rightarrow\infty}\frac{p_{2\ell}(\widehat{G})}{\rho(\widehat{G})^{2\ell}}=2^{|V(\widehat{G})|-|E(\widehat{G})|}.$$ Recall that \begin{equation*}
  \delta=\delta(e)=\begin{cases}
           0, & \mbox{if  $e$ is a pendant edge,}  \\
           1, & \mbox{otherwise}.
         \end{cases}
\end{equation*}
We note that $\widehat{G}$ is the subgraph of $G$ obtained by removing one weakest edge $e$ when $\delta=1$, or by removing  $e$ along with the pendant vertex when $\delta=0$.
So, for ${\widehat{G} \in \mathcal{G}_{\Lambda}}$, we have $|E(\widehat{G})|=|E|-1$ and $|V(\widehat{G})|=|V|-1+\delta$.
Reorganizing \eqref{eq9}, we complete the proof.
\end{proof}

We use the following two lemmas to characterize the multiplicities of points in $\mathbb{V}_{\Lambda}$.

\begin{lem}\cite[Theorem 3.1]{CHEN2021}\label{variety1}
Let $S=\{1,2,\ldots, k-1\}$.
For $i \in S$ and $\mu \neq 0$, let $f_i=\mu x_i^{k-1}-x^{S \setminus \{i\}}$.
Let $\mathcal{V}$ be the affine variety defined by $f_1,\ldots, f_{k-1}$.
Then $\sum_{\mathbf{0} \neq \mathbf{p} \in \mathcal{V}}m_{\mathcal{V}}(\mathbf{p})=k^{k-2}$ and $m_{\mathcal{V}}(\mathbf{0})=(k-1)^{k-1}-k^{k-2}$.
\end{lem}

\begin{lem}\label{variety2}
Let $S=\{1,2,\ldots, k-2\}$.
For $i \in S$ and $\mu \neq 0$, let $f_i=\mu x_i^{k-1}-x^{S \setminus \{i\}}$.
Let $\mathcal{V}$ be the affine variety defined by $f_1,\ldots, f_{k-2}$.
Then $\sum_{\mathbf{0} \neq \mathbf{p} \in \mathcal{V}}m_{\mathcal{V}}(\mathbf{p})=2k^{k-3}$ and $m_{\mathcal{V}}(\mathbf{0})=(k-1)^{k-2}-2k^{k-3}$.
\end{lem}
\begin{proof}

Let the homogeneous polynomial $F_i=\mu x_i^{k-1}-x_0^2x^{S \setminus \{i\}}$ for $i \in S$.
Thus, we have that  $f_i=F_i|_{x_0=1}$.
Let $\overline{F}_i=F_i|_{x_0=0}=\mu x_i^{k-1}$, and note that the homogeneous equations $\overline{F}_1=\cdots=\overline{F}_{k-2}=0$ only have trivial solutions. It implies the affine equations
\begin{align}\label{eqf}
  f_1=\cdots=f_{k-2}=0
\end{align}
has no solutions at infinity, and then all solutions of \eqref{eqf} lie in $\mathbb{C}^{k-2}$. By B\'{e}zout's theorem, we know that the total multiplicity of the solutions of \eqref{eqf} is  $(k-1)^{k-2}$.

For $\mathbf{p}=(p_i) \in \mathcal{V}$, by \eqref{eqf}, we have that $\mu p^{k}_1=\cdots=\mu p^{k}_{k-2}=p^{S}$.
Thus, we have $(p^{S})^{k-2}(\mu^{k-2}(p^{S})^2-1)=0$.
Let $\xi$ be such that $\mu^{k-2}\xi^2-1=0$, indeed $p^{S}= \pm \xi$ for $p^{S}\neq 0$.
Then, for all $i \in [k-2]$, we have $p_i=0$ or $\mu p_i^{k}=\pm \xi$.
In the case of $\mathbf{p} \neq 0 $, for a fixed $i$, we know that $p_i$ has $k$ choices because  $\mu p_i^{k}=\pm \xi$.
In variables $p_{1}, \ldots, p_{k-2}$, there are $k-3$ variables can be chosen freely from the $k$ potential choices, while the remaining one variable has a unique choice to satisfy $p^{S}=\pm \xi$.
Then there are $2k^{k-3}$ distinct non-trivial solutions of \eqref{eqf}.
We will show that every non-trivial solution has multiplicity one, thereby completing the proof.

We use the following fact from multiplicity theory \cite[Page 125]{usingalgebraic}:
if $f_1=\cdots=f_{k-2}=0$ has finitely many solutions and $\mathbf{p}$ is a solution such that the gradient vectors
\begin{align*}
  \nabla f_i(\mathbf{p})=\left(\frac{\partial f_i}{\partial x_1}(\mathbf{p}), \ldots, \frac{\partial f_i}{\partial x_{k-2}}(\mathbf{p}) \right), 1 \leq i \leq k-2
\end{align*}
are linearly independent, then $\mathbf{p}$ is a solution with multiplicity one.
Let $J=(J_{ij})$ be the Jacobian Matrix, i.e., $(J_{i1},J_{i2},\ldots,J_{i k-2})=\nabla f_i(\mathbf{p})$.
Then the entries $J_{ii}=(k-1)\mu p_{i}^{k-2}$ and $J_{ij}=-p^{S \setminus \{i,j\}}$ for $i \neq j$.
Note that $J$ is a $(k-2) \times (k-2)$ square matrix and $|p_i|=|\mu|^{\frac{2-k}{2}}$.
Since $|J_{ii}|=(k-1)|\mu|^{\frac{4-k}{2}} > \sum_{j: j \neq i} |J_{ij}|=(k-3)|\mu|^{\frac{4-k}{2}}$ for all $i \in S$, the matrix $J$ is strictly diagonal dominant.
Then $J$ is non-singular, which implies that the gradient vectors are linearly independent.
\end{proof}

\vspace{2mm}
Now, we are ready to determine $\#\mathbb{V}_{\Lambda}$, and then complete the proof of Theorem \ref{multiplicity}.

\vspace{2mm}

\noindent{\textbf{Part 2 of Theorem \ref{multiplicity}}} $\#\mathbb{V}_{\Lambda}=\sum_{\mathbf{p} \in \mathbb{V}_{\Lambda}} m_{\mathbb{V}_{\Lambda}}(\mathbf{p})=\sum_{\delta=0}^{1}f_{\delta}(k)n_{\delta}$.
\begin{proof}
For each $\mathbf{p}=(p_i) \in \mathbb{V}_{\Lambda}$,  recall that $\hat{\mathbf{p}}$ denote the subvector of $\mathbf{p}$ obtained truncating entry $p_i$ from $\mathbf{p}$ for $i \in supp(\mathbf{p})$.
Thus, there exists $e \in E_w$ such that $\hat{\mathbf{p}} \in \mathbb{V}_{\rho}(G_e^{(k)})$, and we have that $\mathbf{p} = \hat{\mathbf{p}} \oplus \mathbf{0}_{\hat{\mathbf{p}}} $, where $\mathbf{0}_{\hat{\mathbf{p}}} $ is a $(k-1-\delta)$-dimensional zero vector.
Let $\mathcal{V}_0$ and $\mathcal{V}_1$ denote the  affine variety defined in Lemmas \ref{variety1} and \ref{variety2}, respectively.
We note that $\mathbf{0}_{\hat{\mathbf{p}}} \in \mathcal{V}_{\delta}$ for  $\hat{\mathbf{p}} \in  \mathbb{V}_{\rho}(G_e^{(k)})$, and note that $\mathbb{V}_{\Lambda}\cong \bigcup_{e \in E_w} \left(\mathbb{V}_{\rho}(G_e^{(k)})\oplus \mathcal{V}_{\delta}\right)$ since the map in \eqref{eq14} is a bijective map.
Hence, for a fixed $\mathbf{p} \in \mathbb{V}_{\Lambda}$, we have that $m_{\mathbb{V}_{\Lambda}}(\mathbf{p})=m_{\mathbb{V}_{\rho}(G_e^{(k)})}(\hat{\mathbf{p}})m_{\mathcal{V}_{\delta}}(\mathbf{0}_{\hat{\mathbf{p}}})$.
It is following that
\begin{align}
 \#\mathbb{V}_{\Lambda}&=\sum_{\mathbf{p} \in \mathbb{V}_{\Lambda}} m(\mathbf{p}) \notag \\
 &=\sum_{e \in E_w}\sum_{\hat{\mathbf{p}} \in  \mathbb{V}_{\rho}(G_e^{(k)})}m_{\mathbb{V}_{\rho}(G_e^{(k)})}(\hat{\mathbf{p}})m_{\mathcal{V}_{\delta}}(\mathbf{0}_{\hat{\mathbf{p}}}).\label{eq4.8}
\end{align}
Lemma \ref{xinlemma}(b) and Lemma \ref{xinlemma}(c) show that  $m_{\mathbb{V}_{\rho}(G_e^{(k)})}(\hat{\mathbf{p}})=1$ for all ${\hat{\mathbf{p}} \in  \mathbb{V}_{\rho}(G_e^{(k)})}$, and $|\mathbb{V}_{\rho}(G_e^{(k)})|=k^{|E|(k-3)+|V|-k+2-\delta}$.
Lemmas \ref{variety1} and \ref{variety2} show that $m_{\mathcal{V}_{\delta}}(\mathbf{0}_{\hat{\mathbf{p}}})=(k-1)^{k-1-\delta}-2^{\delta}k^{k-3+\delta}$.
Reorganizing \eqref{eq4.8}, we complete the proof.
\end{proof}


\section*{References}
\bibliographystyle{plain}
\bibliography{sebib}
\end{spacing}
\appendix
\section*{Appendix A}
In Appendix A, we present some examples to show that all  candidates outlined in Theorem \ref{thm4} are potential.
Let $\lambda_{\min}(G)=\lambda_n(G)$ for the graph $G$ with $n$ vertices.

For odd-unicyclic graphs, we have $\Lambda(G_1^{(3)})=\sqrt[3]{\rho_V(G_1)^2}=\sqrt[3]{4}$ and $\Lambda(G_2^{(3)})=\sqrt[3]{\lambda_{\min}(G_1)^2}=\sqrt[3]{(2.50578)^2}$.

For bipartite graphs that are not trees, we have $\Lambda(G_3^{(3)})=\sqrt[3]{\rho_V(G_3)^2}=\sqrt[3]{4}$ and $\Lambda(G_4^{(3)})=\sqrt[3]{\rho_{\Gamma}(G_4)^2}=\sqrt[3]{(2.56155)^2}$.

For non-bipartite graphs that are not odd-unicyclic graphs,
we have $\Lambda(G_5^{(3)})=\sqrt[3]{\rho_V(G_5)^2}=\sqrt[3]{9}$; $\Lambda(G_6^{(3)})=\sqrt[3]{\rho_{\Gamma}(G_6)^2}=\sqrt[3]{(2.23607)^2}$; and $\Lambda(G_7^{(3)})=\sqrt[3]{\lambda_{\min}(G_7)^2}=\sqrt[3]{(2.75099)^2}$.
\begin{figure}[H]
\centering
\subfloat[$G_1$]{
\begin{minipage}[t]{0.20\textwidth}
\centering
\includegraphics[scale=0.2]{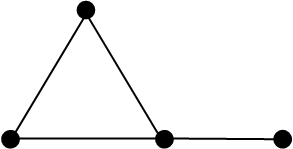}
\end{minipage}
}
\subfloat[$G_2$]{
\begin{minipage}[t]{0.20\textwidth}
\centering
\includegraphics[scale=0.2]{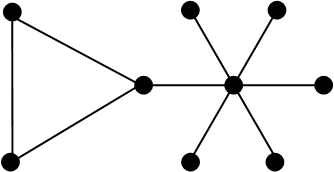}
\end{minipage}
}
\subfloat[$G_3$]{
\begin{minipage}[t]{0.20\textwidth}
\centering
\includegraphics[scale=0.2]{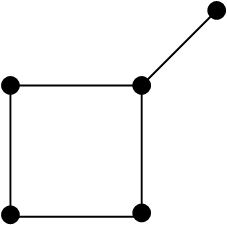}
\end{minipage}
}
\subfloat[$G_4$]{
\begin{minipage}[t]{0.20\textwidth}
\centering
\includegraphics[scale=0.2]{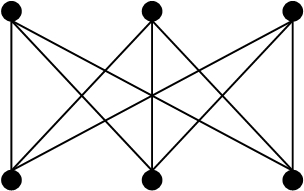}
\end{minipage}
}\\
\subfloat[$G_5$]{
\begin{minipage}[t]{0.20\textwidth}
\centering
\includegraphics[scale=0.2]{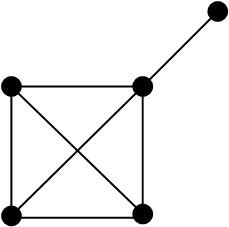}
\end{minipage}
}\subfloat[$G_6$]{
\begin{minipage}[t]{0.20\textwidth}
\centering
\includegraphics[scale=0.2]{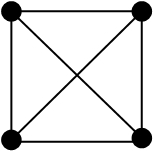}
\end{minipage}
}\subfloat[$G_7$]{
\begin{minipage}[t]{0.20\textwidth}
\centering
\includegraphics[scale=0.2]{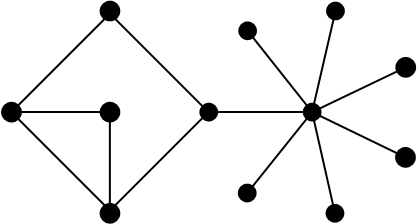}
\end{minipage}
}
\caption{Examples for Theorem \ref{thm4}}
\end{figure}

\end{document}